\documentclass[reqno,12pt]{amsart}
\usepackage[left=2cm,right=2cm,top=1.5cm,bottom=2.5cm]{geometry}

\numberwithin{equation}{section}
\numberwithin{table}{section}
\newtheorem{thm}{Theorem}[section]
\newtheorem{lemm}[thm]{Lemma}
\newtheorem{coro}[thm]{Corollary}

\usepackage{amssymb,color}
\usepackage[hidelinks]{hyperref}

\def\aut{\operatorname{Aut}}

\def\Z{\mathbb Z}
\def\F{\mathbb F}
\def\chr{\operatorname{char}}

\newcommand{\eps}{\varepsilon}

\newcommand{\orth}{\psi}
\newcommand{\nsq}{\zeta}

\renewcommand{\ge}{\geqslant}
\renewcommand{\le}{\leqslant}
\renewcommand{\emptyset}{\varnothing}

\newcommand{\cref}[1]{Corollary~$\ref{#1}$}
\newcommand{\lref}[1]{Lemma~$\ref{#1}$}
\newcommand{\tref}[1]{Theorem~$\ref{#1}$}
\newcommand{\eref}[1]{$(\ref{e#1})$}
\newcommand{\secref}[1]{Section~$\ref{#1}$}
\newcommand{\tabref}[1]{Table~$\ref{#1}$}

\begin{document}
\title[Maximally nonassociative quasigroups]{Maximally nonassociative quasigroups\\ via quadratic orthomorphisms}

\author{{Ale\v s} {Dr\'apal}}
\address{Department of Mathematics\\
Charles University\\
Sokolovsk\'a 83 \\ 186
75 Praha 8\\ Czech Republic}
\email{drapal@karlin.mff.cuni.cz}

\author{{Ian} {M.} {Wanless}}
\address{School of Mathematics \\ 
Monash University \\
Clayton Vic 3800\\
Australia}
\email{ian.wanless@monash.edu}

\begin{abstract}
A quasigroup $Q$ is called \emph{maximally nonassociative} if for
$x,y,z\in Q$ we have that $x\cdot (y\cdot z) = (x\cdot y)\cdot z$
only if $x=y=z$.  We show that, with finitely many exceptions, there exists a
maximally nonassociative quasigroup of order $n$ whenever $n$ is not
of the form $n=2p_1$ or $n=2p_1p_2$ for primes $p_1,p_2$ with $p_1\le
p_2<2p_1$.
\end{abstract}



\maketitle

\section{Introduction}\label{s:intro}

The goal of this paper is to show that for most positive integers $n$
there exists a quasigroup $Q$ of order $n$ such that
\begin{equation}\label{e11}
\forall x,y,z \in Q\colon  \quad 
x\cdot (y\cdot z) = (x\cdot y)\cdot z
\ \ \Longrightarrow \ \
x=y=z. 
\end{equation}
Recall that a \emph{quasigroup} $Q$ is a set with a binary operation,
say $\cdot$, such that the equations $x\cdot a = b$ and $a\cdot y=b$
have unique solutions for all $a,b\in Q$. Quasigroups discussed in
this paper are finite.  There is a natural correspondence between
quasigroups of order $n$ and Latin squares of order $n$.

Quasigroups satisfying \eref{11} are said to be \emph{maximally
nonassociative}. This is because for each quasigroup $Q$ of order $n$
there are at least $n$ triples $(a,b,c)\in Q^3$ such that
$a\cdot (b\cdot c) = (a\cdot b)\cdot c$.  If there are exactly $n$
such triples, then each of them satisfies $a=b=c$, and the quasigroup
is \emph{idempotent}, i.e., $x\cdot x = x$ for each $x\in Q$. This was
shown already in 1980 by Kepka \cite{Kep80}.
Consequently, whenever $Q$ is a maximally nonassociative
quasigroup then it satisfies the reverse implication in \eref{11} as well
(cf. \lref{20} below).
Gro\v sek and Hor\'ak 
\cite{gh} discussed a potential application in cryptography, but
conjectured that maximally nonassociative quasigroups do not exist.
Despite the effort of several authors \cite{gh,Kep80,kr} the existence of
maximally nonassociative quasigroups was not established until 2018,
when an example of order nine was found by a computer
search \cite{dv2}. This was followed by a paper \cite{dl}, in which
Dr\'apal and Lison\v ek proved that maximally nonassociative
quasigroups exist for all orders $p^2$, where $p$ is an odd prime, and
also for the order $64$. The main result of this paper is as follows:

\begin{thm}\label{t:main}
  A maximally nonassociative quasigroup of order $n$ exists for all
  $n\ge 9$, with the possible exception of
  $n\in\{11,12,15,40,42,44,56,66,77,88,90,110\}$ and orders of the form
  $n=2p_1$ or $n=2p_1p_2$ for odd primes $p_1,p_2$ with $p_1\le p_2<2p_1$.
\end{thm}

If $Q_i$ are maximally nonassociative quasigroups of order $n_i$ for
$1\le i \le k$, then the direct product $Q_1\times \dots \times Q_k$
is a maximally nonassociative quasigroup of order $n_1\cdots
n_k$. Therefore our first objective was to try to decide for which odd
primes and which powers of $2$ there exists a maximally nonassociative
quasigroup of such an order. By \cite{dv0,dv1,gh} there exists no
maximally nonassociative quasigroup of order $<9$. The status of order
$11$ is not known.  The existence of maximally nonassociative
quasigroups for each prime order $p\ge 13$ is proved in
\secref{s:Weil}.  

If $n\ge m \ge 3$ and there exists a maximally nonassociative
quasigroup of order $n$, then such a quasigroup also exists for 
order $nm$. This is proved in \secref{s:prod} by means of a specific
product construction.  This construction allows us to develop maximally
nonassociative quasigroups of all orders $2^k$, for $k\ge 4$, from
maximally nonassociative quasigroups of order $16$ and $32$. A
quasigroup for each of the latter two orders is described in
\secref{s:specorth}.

The above mentioned results imply the existence of maximally
nonassociative quasigroups for all but finitely many of the
orders claimed in \tref{t:main}. 
Details are given in \secref{s:specorth}, which also includes
\emph{ad hoc} constructions for the remaining orders.
We believe that in the future a similar construction will be found for
the missing orders $40$, $42$, $44$, $56$, $66$, $77$, $88$, $90$ and
$110$, although it is less clear what will happen for orders $11$,
$12$ and $15$ since they could well be genuine exceptions. We also
suspect that maximally nonassociative quasigroups of orders $2p_1$ and
$2p_1p_2$ will exist for all large enough primes $p_1,p_2$.

With the exception of product constructions, all maximally
nonassociative quasigroups described in this paper were obtained by
using an orthomorphism of an abelian group. An \emph{orthomorphism} of
a group $G$ is a permutation $\orth$ of $G$ such that
$x\mapsto\orth(x)-x$ is also a permutation of $G$. The orthomorphism
$\orth$ is \emph{canonical} if $\orth(0) = 0$, where we use $0$ to
denote the identity element, since our groups are always
abelian. Later we will use the observation that $0$ is the only
fixed point of a canonical orthomorphism.
Orthomorphisms have been used in many different situations
for creating interesting quasigroups and Latin squares. See
\cite{Eva18,diagcyc} for surveys. From any orthomorphism $\orth$ of $G$
we can define a quasigroup operation $*$ on $G$ by
\begin{equation}\label{e:quasiorth} 
x*y = x + \orth(y-x)
\end{equation}
for all $x,y \in G$. With the exception of \secref{s:specorth}, our $G$ will
be the additive group of a finite field $\F=\F_q$ of odd order $q$. 
In that case there are \emph{quadratic orthomorphisms} available,
namely orthomorphisms defined by
\begin{equation}\label{e:quadorth}
\orth (x) = 
\begin{cases} 
  ax &\text{if $x$ is a square,} \\
  bx &\text{if $x$ is a nonsquare,}
\end{cases}
\end{equation}
where $a,b$ are fixed elements of $\F$.  Note that $x\in\F$ is called
a \emph{square} if it can be expressed as $x=y^2$ for some
$y\in\F$. The other elements are \emph{nonsquares}.

If $\psi$ is the orthomorphism defined by \eref{:quadorth}, then the
quasigroup defined by \eref{:quasiorth} will be denoted by
$Q_{a,b}$. These quasigroups will play a central role in this
paper. Petr Lison\v ek
\cite{Lis20} has independently and concurrently obtained some of the
results in our paper, also by using quadratic orthomorphisms.  The
following basic properties of quadratic orthomorphisms are known, see
\cite{Eva18,cyclatom}.

\begin{lemm}\label{l:basicquad}
For \eref{:quadorth} to define a canonical orthomorphism of $\F_q$ 
it is necessary and sufficient
that $ab$ and $(a-1)(b-1)$ are both nonzero squares. Assuming that 
\eref{:quadorth} does define an orthomorphism, the resulting quasigroup
$Q_{a,b}$ has the following properties:
\begin{itemize}
\item[(i)] $Q_{a,b}$ is idempotent.
\item[(ii)] For any $f\in\F$ the map $x\mapsto x+f$ is an automorphism of 
$Q_{a,b}$. 
\item[(iii)]
  For any nonzero square $c\in\F$ the map $x\mapsto cx$ is an automorphism of 
$Q_{a,b}$. 
\item[(iv)]  
$Q_{a,b}$ is isomorphic to $Q_{b,a}$ by the map $x\mapsto\nsq x$, where
$\nsq$ is any nonsquare in $\F_q$.
\item[(v)] 
The opposite quasigroup of $Q_{a,b}$ is $Q_{1-a,1-b}$ if $q\equiv1\bmod4$ and
$Q_{1-b,1-a}$ if $q\equiv3\bmod4$.
\end{itemize}
\end{lemm}

Note that the \emph{opposite} quasigroup $(Q,\cdot)$ of a quasigroup $(Q,*)$
is the quasigroup satisfying $a\cdot b=b*a$ for all $a,b\in Q$. In other words,
the opposite quasigroup is obtained by transposing the operation table.

\begin{lemm}\label{11}
Let $\orth$ be a canonical orthomorphism of an abelian group $(G,+)$. 
The idempotent quasigroup
$(G,*)$ defined by \eref{:quasiorth}
is maximally nonassociative if and only if there are no 
$x,y \in G\setminus \{0\}$ such that 
\begin{equation}\label{e:assoc}
\orth(\orth(x)+y)-\orth(y) = \orth(x+y-\orth(y)).
\end{equation}
If $\orth\in \aut(G)$, then $(x*y)*z = x*(y*z)$ if and only 
if $x=z$.
\end{lemm}
\begin{proof} 
We have $x*(y*z) =x + \orth((y*z)-x) = x + \orth((y-x)+\orth(z-y))$ 
and $(x*y) *z = (x + \orth(y-x))*z= 
x+\orth(y-x) +\orth((z-x)-\orth(y-x))$. Thus
$(x*y)*z = x*(y*z)$ if and only if $\orth(\orth(v) + u)-\orth(u)
= \orth(u+v-\orth(u))$, where $u = y-x$ and $v = z-y$. 
If $\orth\in \aut(G)$, then this is true if and only if $\orth^2(v) 
= \orth(v) + \orth(u-\orth(u))$, which is equivalent to
$\orth(v)-v=u-\orth(u)$ and hence also to $\orth(u+v)=u+v$. 
This last condition holds if and only if $u+v=0$,
i.e.,~$x=z$.

To finish the proof note that we have already shown that
$(G,*)$ is maximally nonassociative if and only if \eref{:assoc}
holds exactly when $x=y=0$. However, if $x=0$ and \eref{:assoc}
holds, then $y=0$ because $0 = \orth(y-\orth(y))$. Meanwhile
$y=0$ forces $x=0$ since in such a case \eref{:assoc} reduces
to $\orth^2(x) = \orth(x)$. It therefore suffices to test \eref{:assoc}
for nonzero $x,y$.
\end{proof}

If $G$ is of order $n$, then it might seem that $(n-1)^2$ tests are needed
to verify \eref{:assoc}. However, as we will formalise in \lref{31}, the
number of tests can be reduced dramatically given the large number of
automorphisms of $Q_{a,b}$ that we have at our disposal.

We say that a list of polynomials $g_1,\dots,g_k$ with coefficients in
$\F$ \emph{avoids squares} if there exists no sequence $1\le
i_1<\dots<i_r \le k$ such that $r\ge 1$ and $g_{i_1}\cdots g_{i_r}$ is
a square (as a polynomial with coefficients in the algebraic closure
$\overline\F$ of $\F$).
Define $\chi\colon \F\to \{\pm1,0\}$ to be the quadratic
character extended by $\chi(0)=0$.  The following consequence of the
Weil bound will be used several times: 

\begin{thm}\label{t:Weil}
Let $g_1,\dots,g_k\in \F[t]$ be a list of polynomials that avoids squares.
Suppose for $1\le i\le k$ that $g_i$ has degree $d_i\ge1$ and
that $\eps_i\in\{-1,1\}$. Denote by
$N$ the number of all $\alpha \in \F$ such that 
$\chi(g_i(\alpha))= \eps_i$, for all $1\le i \le k$.
Then
$|N-2^{-k}q| \le 
(\sqrt q+1)D/2-\sqrt{q}(1-2^{-k})$ where $D=\sum_i d_i$.
\end{thm}

\begin{proof} Consider
\[ 
\sum_{\alpha\in \F}\prod_{1\le i \le k}
\big(1+\eps_i\chi(g_i(\alpha))\big)=2^k N+R
\]
where $R$ is the contribution to the left hand side from all $\alpha$
that are roots of at least one of the $g_i(\alpha)$. We have 
$|R|\le2^{k-1}D$ because $D=\sum_i d_i$ is an upper bound on the number
of $\alpha$ that contribute to $R$.

On the other hand, exploiting the multiplicative nature of $\chi$ we have
\begin{align*} 
\sum_{\alpha\in \F}\prod_{1\le i \le k}
\big(1+\eps_i\chi(g_i(\alpha))\big)
&=\sum_{\alpha \in \F}\biggl(1+\sum_{U}\Big(\prod_{i\in U}\eps_i\Big)
\chi\Big(\prod_{i\in U}g_{i}(\alpha)\Big)\biggr)\\
&=q+\sum_{U}\Big(\prod_{i\in U}\eps_i\Big)
\sum_{\alpha \in \F}\chi\Big(\prod_{i\in U}g_{i}(\alpha)\Big),
\end{align*}
where $U$ runs over all nonempty subsets $U\subseteq\{1,2,\dots,k\}$.
Therefore
\begin{align*}
|N-2^{-k}q| &\le 2^{-k}|R|+ 2^{-k}\sum_{U}
\Big|\sum_{\alpha\in\F} \chi\Big(\prod_{i\in U}g_{i}(\alpha)\Big)\Big|\\
&\le D/2+2^{-k}{\sqrt q}\sum_{U}\Big(\sum_{i\in U}d_i-1\Big)\\
&= D/2+2^{-k}{\sqrt q}\Big(2^{k-1}D-2^k+1\Big)\\
&=(\sqrt q+1)D/2-\sqrt{q}(1-2^{-k}),
\end{align*}
the last inequality being the application of the Weil bound
as formulated in \cite[Theorem 6.2.2]{ARL}.
\end{proof}

\begin{coro}\label{cy:Weil}
Under the hypotheses of the theorem, if
$2^{-k}q>(\sqrt q+1)D/2-\sqrt{q}(1-2^{-k})$ then $N>0$.
\end{coro}

The structure of the paper is as follows. In \secref{s:prod} we define
a particular product construction that allows us to build larger
maximally nonassociative quasigroups from smaller ones.
Sections~\ref{s:quad} and~\ref{s:Weil} investigate the quasigroup
$Q_{a,b}$. The first aim is to explain that $Q_{a,b}$ is maximally
nonassociative unless $a$ and $b$ satisfy a number of conditions, each
of which stipulates that several polynomials (two or three) in $a$ and
$b$ yield a square (in some situations) or a nonsquare (in other
situations). An argument based on \cref{cy:Weil} is then used in
\secref{s:Weil} to show that for all large enough primes $p$
there exists at least one pair $(a,b)\in\Z_p^2$ for which none of the
conditions is satisfied. The strategy used in
Sections~\ref{s:quad}--\ref{s:Weil} thus mimics that of \cite{dl}.

Orthomorphisms that yield maximally nonassociative
quasigroups of orders $16$, $20$, $21$, $24$, $28$, $32$, $33$, $35$ and 
$55$ are listed in \secref{s:specorth}. These provide the last piece of the
proof of \tref{t:main}, which is given at the end of that section.

\section{The product construction}\label{s:prod}
This section uses a standard convention of quasigroup
theory, by which a juxtaposition is of higher precedence than
an explicitly stated operation. Thus $xy\cdot z = (x\cdot y)\cdot z$.
A triple $(x,y,z)$ is \emph{associative} if and only if
$xy\cdot z = x\cdot yz$. 

Let us begin by mentioning two easy and well known facts of a general
nature.

\begin{lemm}\label{20}
Let $Q$ be a quasigroup satisfying \eref{11}. Then $Q$ is idempotent.
\end{lemm}

\begin{proof}
Let $y\in Q$ and define $x,z\in Q$ by $xy=y=yz$. Then
$x\cdot yz=xy=y=yz=xy\cdot z$. Hence, $x=y$ by \eref{11}, and so $yy=y$.
As $y$ was arbitrary, $Q$ must be idempotent.
\end{proof}

\begin{lemm}\label{21}
Let $Q$ be an idempotent quasigroup, and let $(x,y,z)\in Q^3$
be an associative triple. If $x=y$ or $y=z$ or $xy = z$ or $x=yz$,
then $x=y=z$.
\end{lemm}
\begin{proof} If $x=y$, then $xz=xy\cdot z=x\cdot yz = x\cdot xz$. By 
cancellation, $xz = z = zz$ and $x=z$. If $x=yz$,
then $yz= x = xx =x\cdot yz = xy \cdot z$. By cancellation,
$xy = y =yy$ and $x=y$. The rest follows by mirror arguments.
\end{proof}

Although we are exclusively interested in finite quasigroups in this
paper, we note in passing that both of the previous results apply
when $Q$ is infinite.

\begin{thm}\label{22}
Let $(Q,\cdot)$ be a maximally nonassociative finite quasigroup and let
$(U,*)$ be an idempotent quasigroup. 
Suppose that $|Q|\ge|U|$ so that there exists an injective
mapping $j\colon U\to Q$.
Choose an abelian group operation on $Q$, and denote it 
by $+$. Then
\begin{equation}\label{e21}
(x,\,u)(y,\,v) = \begin{cases}
(x\cdot y,\,u) &\text{if $u=v$, and} \\
(x+y+j(u),\, u*v) &\text{if $u\ne v$} \end{cases}
\end{equation}
defines a maximally nonassociative quasigroup operation
on $Q\times U$.
\end{thm}

\begin{proof} First note 
that $\pi\colon (x,u) \mapsto u$ is a homomorphism onto $(U,*)$.
By applying $\pi$ we see that to show 
that $(x_1,u_1)(y,v) = (x_2,u_2)(y,v)$ implies both $x_1=x_2$ and
$u_1=u_2$, only the case $u_1=u_2=u$ has to be treated. 
If $u=v$, then $x_1=x_2$ since $(Q,\cdot)$ is a quasigroup.
Assume $u\ne v$. Then $x_1+y+j(u)=x_2+y+j(u)$ and $x_1=x_2$.
We have thus verified cancellation on the left,
and cancellation on the right can be verified in a similar manner.
This means that \eref{21}
defines a quasigroup. Note that this fact does not depend upon 
the injectivity of $j$.

Now consider an associative triple $((x,u),(y,v),(z,w))$.
If $u=v=w$, then $x=y=z$, by the assumption on
$(Q,\cdot)$. The triple $(u,v,w)$ is also associative since
$\pi$ is a quasigroup homomorphism. Hence it may be assumed
that $u\ne v$, $u*v \ne w$, $v\ne w$ and $u\ne v*w$, by \lref{21}.
This means that
\begin{align*}
(x,u)\cdot(y,v)(z,w) &= \bigl (x + y + z + j(u) + j(v),\,u*(v*w)\bigr ), 
\text{ while}\\
(x,u)(y,v)\cdot(z,w) &= \bigl (x + y + z + j(u) + j(u*v),\, (u*v)*w\bigr ).
\end{align*}
The associativity of the triple $((x,u),(y,v),(z,w))$
thus yields $j(v) = j(u*v)$, which is the same as $u=v$
since $j$ is injective and $v = v*v$.
Hence no nondiagonal associative triples exist.
\end{proof}

\begin{coro}\label{cy:prod}
If $n\ge m \ge 3$ are integers, and there exists
a maximally nonassociative quasigroup of order $n$, then
there exists a maximally nonassociative quasigroup of order $nm$.
\end{coro}

\begin{proof}
This follows directly from \tref{22} since for each order $m\ge 3$ it
is well known that there exists an idempotent quasigroup of that
order.
\end{proof}

\section{Quadratic orthomorphisms}\label{s:quad}

Let $\F$ be a finite field of odd order $q$.
Denote by $\Sigma$ the set of all $(a,b)\in \F\times \F$
such that $0\notin\{a,b,a-1,b-1\}$, $a \ne b$, and both
$ab$ and $(a-1)(b-1)$ are squares. 
By Lemma~\ref{l:basicquad} each pair $(a,b)\in \Sigma$ \emph{induces}
a quasigroup $Q_{a,b}$.
The condition $a\ne b$ has been included in the definition of $\Sigma$,
since $Q_{a,a}$ is not maximally nonassociative.
Indeed, in this case each triple $(x,y,x)\in \F^3$ is associative,
by \lref{11}.

Replacing $y$ with $-y$ in \eref{:assoc} yields 
\begin{equation}\label{e:assoc-}
\orth(\orth(x)-y) = \orth(-y) + \orth(x-y-\orth(-y)).
\end{equation}
In this section \eref{:assoc-} is given preference over \eref{:assoc} since
it makes the connection to the opposite quasigroup easier to handle.
Equation \eref{:assoc-} 
will henceforth be called the \emph{Associativity Equation}.
By \lref{11}, the quasigroup $Q_{a,b}$ is maximally
nonassociative if and only if the Associativity Equation
has no solution $(x,y)\ne (0,0)$, where $\orth$ is defined by
\eref{:quadorth}. Our next result will reduce
the number of tests required to check if there is a solution. It will
show that it suffices to test just two values of $x$ provided one of
those values is a square and the other is not.

\begin{lemm}\label{31}
For $(a,b)\in \Sigma$ define $\orth$ by \eref{:quadorth} and 
$*$ by \eref{:quasiorth}. Then
\begin{equation}\label{e32}
x-y-\orth(-y) = x - (y*0) \quad\text{and}\quad \orth(x)-y = (0*x)-y.
\end{equation}
An ordered pair $(x,y)\in \F^2$ fulfils \eref{:assoc-} if and only
if $y*(0*x) = (y*0)*x$. Furthermore,
if $(x,y)\ne (0,0)$ fulfils \eref{:assoc-}, then none of 
$x$, $y$, $x-y-\orth(-y)$ and $\orth(x)-y$ vanishes,
and $(cx,cy)$ fulfils \eref{:assoc-} too, for any square $c\in \F$. 
\end{lemm}

\begin{proof} By definition, $0*x = \orth(x)$ and 
$y*0 = y+\orth(-y)$, which yields \eref{32}. We have
\[ (y*0)*x = y +\orth(-y) + \orth(x-y-\orth(-y)) \text{ and }
y*(0*x) = y + \orth(\orth(x)-y).\]
Hence $(x,y)$ fulfils \eref{:assoc-} if and only if $(y,0,x)$ is
an associative triple in $Q_{a,b}$. Assume $(x,y)\ne (0,0)$ 
and suppose that $(y,0,x)$ is an associative triple. Then
$0\notin \{x,y\}$, $y+\orth(-y) \ne x$ and $y\ne \orth(x)$, by
\lref{21}. To conclude, note that $(y,0,x)$ is an associative
triple if and only if $(cy,0,cx)$ is an associative triple,
for any square $c\in \F$, by \lref{l:basicquad}(iii). 
\end{proof}

Define $\eta:\F\to\{0,1\}$ by $\eta(x) = 0$ if
$x$ is a square, and $\eta(x) = 1$ if $x$ is a nonsquare.
As illustrated by \eref{33}, the Associativity Equation \eref{:assoc-}
takes different shapes depending upon the values of $\eta(x)$ and
$\eta(-y)$:
\begin{equation}\label{e33}
\begin{array}{ccccc}
\eta(x) & \eta(-y) &\orth(x) & \orth(-y)& \text{The Associativity Equation}\\ 
\hline \\ [-0.6em]
0 & 0 & ax & -ay &\orth(ax-y) = \orth(x-y+ay)-ay\\
0 & 1 & ax & -by &\orth(ax-y) = \orth(x-y+by)-by\\
1 & 0 & bx & -ay &\orth(bx-y) = \orth(x-y+ay)-ay \\
1 & 1 & bx & -by &\orth(bx-y) = \orth(x-y+by)-by
\end{array}
\end{equation}

The shape of the Associativity Equation depends not only upon
$\eta(x)$ and $\eta(-y)$, but also upon $\eta(\orth(x)-y)$ and
$\eta(x-y-\orth(-y))$:
\begin{equation}\label{e34}
\begin{array}{ccc}
\eta(\orth(x)-y) & \eta(x-y-\orth(-y)) & 
\text{The Associativity Equation}\\ 
\hline \\ [-0.6em]
0 & 0 & a(\orth(x)-y)= \orth(-y) + a(x-y-\orth(-y))\\
0 & 1 & a(\orth(x)-y)= \orth(-y) + b(x-y-\orth(-y))\\
1 & 0 & b(\orth(x)-y)= \orth(-y) + a(x-y-\orth(-y))\\
1 & 1 & b(\orth(x)-y)= \orth(-y) + b(x-y-\orth(-y))
\end{array}
\end{equation}

For $(a,b)\in \Sigma$ and $i,j,r,s\in \{0,1\}$ denote by 
$E_{ij}^{rs}(a,b)$ the set
of all nontrivial solutions $(x,y)$ to the Associativity Equation 
\eref{:assoc-} that fulfil 
\begin{equation}\label{e35}
i = \eta(x),\ j=\eta(-y),\ r=\eta(\orth(x)-y)
\text{ and } s=\eta(x-y-\orth(-y)).
\end{equation}

Put $\rho = (a-1)/(b-1)$ and
consider the case $(i,j,r,s)=(0,1,1,1)$ as an example.
Note that $\rho$ is a square, by the definition of $\Sigma$. 
By \eref{34}, the Associativity Equation takes the 
form $b(\orth(x)-y)= \orth(-y) + b(x-y-\orth(-y))$.
By \eref{33}, $\orth(x)$ is to be replaced by $ax$ and
$\orth(-y)$ by $-by$. The Associativity Equation thus yields
\begin{equation}\label{e36}
\begin{aligned} 
bax-by &= -by + bx - by + b^2y, \\
x(a-1) &= y(b-1), \text{ and} \\
y  &= \rho x.
\end{aligned}
\end{equation}
Any solution to \eref{36} thus fulfils $\eta(x) = \eta(y)$. This means that
if $-1$ is a square, then we cannot have $\eta(x)=0$ and $\eta(-y)=1$.
Therefore if $-1$ is a square, then $E_{01}^{11}=\emptyset$.
Let $-1$ be a nonsquare, and $x$ a square. Then $(x,\rho x)\in
E_{01}^{11}$ if $ax-\rho x$ is a nonsquare, i.e.~if $\rho - a$
is a square, and if $x-\rho x +b\rho x$ is a nonsquare, i.e.~if
$\rho - 1 -b\rho = (1-b)\rho - 1=-a$ is a square.

Since $(1,\rho)$ is a solution of \eref{36}, then every
solution to \eref{36} is equal to $(x,\rho x)$ for some $x\in \F$.
However, if $x$ is a nonsquare, then this solution does not
fulfil \eref{35}. Hence
\[
E_{01}^{11}(a,b) =\begin{cases} \big\{(c^2,c^2\rho);\ c\in \F^*\big\} 
&\text{ if $-1$ and $a$ are nonsquares and $\rho - a$
is a square,}\\
\;\emptyset&\text{ otherwise.}
\end{cases}
\]

The equations \eref{36} have been obtained by combining the row $(r,s) = (1,0)$
of \eref{34} with the row $(i,j)=(0,1)$ of \eref{33}. There
are 16 combinations altogether. However, the workload in studying these
different combinations can be reduced by the following observations.

\begin{lemm}\label{l:comp}
Let $(a,b)\in \Sigma$ and suppose that $\nsq$ is a nonsquare. Then
\[
(x,y)\in E_{ij}^{rs}(a,b) \ \Longleftrightarrow \
(\nsq x,\nsq y) \in E_{1-i,1-j}^{1-r,1-s}(b,a),\]
for all $i,j,r,s\in \{0,1\}$.
\end{lemm}

\begin{proof} 
Let $\orth$ and $\bar\orth$ denote the orthomorphisms for 
$(Q_{a,b},*)$ and $(Q_{b,a},\bar *)$, respectively.
By \lref{31}, $(x,y)$ fulfils the Associativity Equation \eref{:assoc-}
if and only if $(y,0,x)$ is an associative triple in $Q_{a,b}$.
By \lref{l:basicquad}(iv) we know that $x\mapsto \nsq x$ is an 
isomorphism from $Q_{a,b}$ to $Q_{b,a}$. Hence,
\[
y*(0*x) = (y*0)*x \ \Longleftrightarrow \
(\nsq y)  \,\bar*\,  (0  \,\bar*\,  (\nsq x)) =
((\nsq y) \,\bar*\,  0)  \,\bar*\,  (\nsq x),\]
for all $x,y\in \F$. Therefore $(x,y)$ fulfils \lref{31} with respect
to $*$ if and only if $(\nsq x,\nsq y)$ fulfils \lref{31} with respect to
$\bar *$.

Let $x,y\in \F$ and $i,j,r,s \in \{0,1\}$ be such that
$(x,y) \in E_{ij}^{rs}(a,b)$. Note that $\orth(x)-y\ne 0$
and $x-y-\orth(-y)\ne 0$, by \lref{31}. Also,
$\nsq(\psi(x)-y) = \nsq(0*x)-\nsq y=
(0\,\bar *\, \nsq x) -\nsq y=\bar\orth(\nsq x)-\nsq y$ and
$\nsq(x-y-\orth(-y))
=\nsq x-\nsq(y*0)
=\nsq x - ((\nsq y) \,\bar *\, 0)
=\nsq x - \nsq y -\bar \orth (-\nsq y)$. It follows that
\[\frac{\nsq x}x =\frac{\nsq y}{y} =
\frac{\bar \orth(\nsq x)-\nsq y}{\orth(x)-y} =
\frac{\nsq x - \nsq y - \bar \orth(-\nsq y)}{x-y-\orth(-y)} = \nsq.\]
This verifies that
$(\nsq x,\nsq y) \in E_{1-i,1-j}^{1-r,1-s}(b,a)$ given our earlier observations.
\end{proof}

\begin{lemm}\label{l:opp}
If $(a,b)\in \Sigma$ and $i,j,r,s\in \{0,1\}$, then 
\begin{align*}
&(x,y)\in E_{ij}^{rs}(a,b) \ \Longleftrightarrow \
(y,x) \in E_{ji}^{sr}(1-a,1-b) \text{ if $-1$ is a square, and} \\
&(x,y)\in E_{ij}^{rs}(a,b) \ \Longleftrightarrow \
(y,x) \in E_{1-j,1-i}^{1-s,1-r}(1-b,1-a) \text{ if $-1$ is a nonsquare.}
\end{align*}
\end{lemm}

\begin{proof} Let $*$ and $\tilde*$ denote the operations
of $Q_{a,b}$ and the opposite of $Q_{a,b}$, respectively
(see \lref{l:basicquad}(v)).
Define $\orth$ by \eref{:quadorth} and
put $\tilde \orth(x) = x + \orth(-x)$ for each $x\in \F$. 
Then $\tilde\orth$ is an orthomorphism of $(\F,+)$, and 
the operation $\tilde*$ is defined by
$x\,\tilde *\,  y = x+\tilde \orth (y-x)$ 
for all $x,y\in \F$. This is because
$$x \,\tilde *\,  y = 
y*x  = y+\orth(x-y) = x + (y-x) + \orth(-(y-x)) =
x + \tilde\orth(y-x),$$ 
for $x,y \in \F$.

Since $y*(0*x) = (y*0)*x$ is equivalent to 
$x\,\tilde *\,(0  \,\tilde *\, y) = (x\,\tilde *\,  0) \,\tilde *\,  y$,
a pair $(x,y) \in \F\times \F$ fulfils the Associativity Equation 
\eref{:assoc-} with respect to $\orth$ if and only if the equation
is fulfilled by $(y,x)$ with respect to $\tilde \orth$.

Let $i,j,r,s\in \{0,1\}$ and $x,y \in \F$ be such that
$i=\eta(x)$, $j=\eta(-y)$, $r=\eta(\orth(x)-y)$ and
$s = \eta(x-y-\orth(-y))$. Define
$(\tilde \imath, \tilde \jmath,\tilde r,\tilde s)$ to be
$(j,i,s,r)$ if $-1$ is a square, and to be $(1-j,1-i,1-s,1-r)$
if $-1$ is a nonsquare. What remains is to verify that
$\tilde \imath = \eta(y)$, $\tilde
\jmath = \eta(-x)$, $\tilde r = \eta(\tilde \orth(y)-x)$ and
$\tilde s = \eta(y-x-\tilde\orth(-x))$. Now, $\tilde \imath = \eta(y)$
and $\tilde\jmath = \eta(-x)$ follow immediately from the definition
of $\tilde \imath$ and $\tilde \jmath$. As for $\tilde r$ and $\tilde s$,
observe that
\begin{align*}
\tilde \orth(y) - x &= y+\orth(-y) - x = -(x-y-\orth(-y)), \text{\ and}\\
y-x-\tilde \orth(-x) &= y-x-(-x+\orth(x)) = -(\orth(x)-y).\qedhere
\end{align*}
\end{proof}

We need one further technical lemma before stating the main result of
this section.

\begin{lemm}\label{l:tech}
Suppose that $(a,b)\in\Sigma$. Then at least one of the following holds:
\begin{enumerate}
\item[(i)] $b\ne a^2$, 
\item[(ii)] $a,b,1-a$ and $1-b$ are all squares, or
\item[(iii)] there exists $y\in\F$ such that 
$\eta(-y)=1$, $\eta(a-y)=0$ and $\eta(1-y+by)=1$.
\end{enumerate}
\end{lemm}

\begin{proof}
As $(a,b)\in\Sigma$ we know that $\eta(a)=\eta(b)$ and
$\eta(1-a)=\eta(1-b)$. If condition (i) fails then $b$ is a square, so
$\eta(a)=\eta(b)=0$. If condition (ii) also fails then
$\eta(1-a)=\eta(1-b)=1$ and hence $a\ne1/(1-b)$. It then follows that
the list of linear polynomials $-y,a-y,1-(1-b)y$ avoids squares since
no pair of them have a root in common. Thus we may apply \cref{cy:Weil}
with $k=D=3$ to find that condition (iii) holds provided
$q/8>(\sqrt{q}+1)3/2-\sqrt{q}(7/8)$. This proves the lemma for all
$q\ge 46$.  For smaller fields the lemma can be checked by direct
computation.
\end{proof}

With the preliminary results in place, we can now characterise the
quadratic orthomorphisms that produce maximally nonassociative
quasigroups.

\begin{thm}\label{t:charnonass}
For $(a,b)\in\Sigma$, define $\mu=b^2-2b+a$, $\nu=a^2-2a+b$, 
$\sigma=a^2b-a^2-ab+b$, and $\tau=a^2b-ab-a+b$.
The necessary and sufficient conditions for $Q_{a,b}$ to be maximally 
nonassociative are
\begin{itemize} 
\item[(1)] $a^2\ne b$ or $a\ne 2b-b^2$,

\item[(2)] at least one of $-1$, $a-1$ or $a$ is nonsquare,

\item[(3)] at least one of $b$, $(1-a)(a^2-b)$ or $\sigma(a-1)$ is square,

\item[(4)] at least one of $a\nu$, $1-b$ or $a\tau$ is square,

\item[(5)] $-1$ is nonsquare or $\sigma a(b-1)$ is square or $\tau a(b-1)$ is square,

\item[(6)] $-1$ is square or $b-1$ is nonsquare or $(ab-a+b)b$ is nonsquare,

\item[(7)] $(b-a^2)\mu$ is square or $b\mu(ab-2a+1)$ is nonsquare or $(a-1)(ab-a+b)\mu$ is square,

\item[(8)] $-1$ is square or $a-1$ is square or $b$ is nonsquare,

\item[(9)] at least one of $-1$, $a$ or $(ab-2a+1)(b-1)$ is square, and

\item[(10)] conditions $(1)-(9)$ all apply when $a$ and $b$ are interchanged
$($which also interchanges $\mu$ with $\nu$, and changes
$\sigma$ and $\tau$ accordingly$)$.

\end{itemize}

\end{thm}

\begin{proof}
We consider the 8 possibilities for the quadruple $ijrs$ with $i=0$. The 8
possibilities with $i=1$ can then be obtained by employing
\lref{l:comp}, and will lead to the same restrictions but with $a$ and
$b$ interchanged. Note that the case $ijrs=0111$ was already worked
through in some detail before \lref{l:comp}.

For specific values of $ijrs$ we can use \eref{33} and \eref{34}
to convert the Associativity Equation \eref{:assoc-} into a linear
equation in $x$ and $y$. As we are assuming that $0=i=\eta(x)$ we may
then without loss of generality substitute $x=1$, by
\lref{31}. In this way, for each of the 8 cases, the
Associativity Equation \eref{:assoc-} reduces to the form given in
\tabref{tab:asseq}.  Common factors of $a$, $b$, $1-a$ or $1-b$ have
been cancelled from both sides of the Associativity Equation if they
were present. These quantities are assumed to be nonzero since
$(a,b)\in\Sigma$.

\begin{table}[h]
\begin{center}
$\begin{array}
{rll}
ijrs & \text{The Associativity Equation} & \text{Solution }y_{ij}^{rs}  \\[0.3em]
\hline \\ [-0.6em] 
0000 & 1 = y & 1\\ 
0001 & a^2-b =b(a-1)y & (a^2-b)/b(a-1) \\ 
0010 & a(b-1) = (a^2-2a+b)y & a(b-1)/\nu\\ 
0011 & b(a-1) = a(b-1)y  & (a-1)b/(b-1)a \\
0100 & a= by & a/b \\ 
0101 & a^2-b = (b^2-2b+a)y & (a^2-b)/\mu\\ 
0110 & 1 = y & 1\\ 
0111 & a-1 = (b-1)y & (a-1)/(b-1) \\
\end{array}$
\end{center}
\caption{\label{tab:asseq}}
\end{table}

If the coefficient of $y$ in the Associativity Equation is nonzero,
then there is a unique solution, denoted by $y=y_{ij}^{rs}$ as listed
in the rightmost column of \tabref{tab:asseq}. There are two of the 8
cases where the coefficient of $y$ may be zero depending on the values
of $a$ and $b$.  When $ijrs=0010$ we have $\nu y=a(b-1)\ne0$, so
there will be no solution if $\nu=0$ and a unique solution
otherwise.  

The case $ijrs=0101$ needs more care because it is possible that both
sides of the Associativity Equation are zero if $\mu=0$ and $b=a^2$.
If that happens then any $y$ will be a solution and $Q_{a,b}$ will not
be maximally nonassociative (note that by \lref{l:tech} we can assume
that a suitable $y$ exists or that one of conditions (2) or (8)
of the theorem fails). If precisely one of the
conditions $\mu=0$ and $b=a^2$ holds, then there is no solution to
the Associativity Equation (we are assuming $y\ne0$ by \lref{11}). If 
$\mu\ne0$ and $b\ne a^2$, there is a unique solution $y_{01}^{01}=(a^2-b)/\mu$
as presented in \tabref{tab:asseq}.

The interpretation of \tabref{tab:asseq} is that in each case there
will be no nondiagonal associative triples unless substituting $x=1$
and $y=y_{ij}^{rs}$ into \eref{35} produces the correct values for
$i,j,r,s$ for the case in question. The condition $i=\eta(x)$ is
automatically satisfied but the other three conditions produce
restrictions.  These restrictions can be simplified using
$\eta(ab)=\eta((1-a)(1-b))=\eta(c^2)=0$ for all $c\in\F$, producing
\tabref{tab:conds}.

\begin{table}[h]
\begin{center}
$\begin{array}
{r l l l}
ijrs & j=\eta(-y) & r=\eta(a-y) & s=\eta(x-y-\orth(-y))\\
[0.3em]
\hline \\ [-0.6em] 
0000 & \eta(-1)=0&\eta(a-1)=0&\eta(a)=0\\ 
0001 & \eta((a^2-b)b(1-a))=0&\eta(\sigma b(a-1))=0&\eta(b)=1 \\ 
0010 & \eta(a\nu(1-b))=0&\eta(a\nu)=1&\eta(\tau\nu)=0\\ 
0011 & \eta(-1)=0&\eta(\sigma a(b-1))=1&\eta(\tau a(b-1))=1 \\
0100 & \eta(-1)=1&\eta(b-1)=0&\eta((ab-a+b)b)=0 \\ 
0101 & \eta((b-a^2)\mu)=1&\eta(b(ab-2a+1)\mu)=0&\eta((a-1)(ab-a+b)\mu)=1\\ 
0110 & \eta(-1)=1&\eta(a-1)=1&\eta(b)=0\\ 
0111 & \eta(-1)=1&\eta((ab-2a+1)(b-1))=1&\eta(a)=1 \\
\end{array}$
\end{center}
\caption{\label{tab:conds}}
\end{table}

If for any row of \tabref{tab:conds} all three conditions are met,
then $(y_{ij}^{rs},0,1)$ will be an associative triple, by
\lref{31}. So we are interested in the case when at least one
condition fails in each row of \tabref{tab:conds}.  For most rows this
translates directly to a condition in the theorem.  For the case
$ijrs=0001$ we consider the $\eta(b)=0$ and $\eta(b)=1$ cases
separately to get the condition that at least one of
$b$, $(1-a)(a^2-b)$ or $\sigma(a-1)$ is square.  Similarly, for the case
$ijrs=0010$ we consider the $\eta(a\nu)=0$ and $\eta(a\nu)=1$
cases separately to get the condition that at least one of $a\nu$,
$1-b$ or $a\tau$ is square.  Note that this covers the case when
$\nu=0$ and there was no solution to the Associativity Equation,
so that condition does not need separate treatment. Similarly, in the
subcase of $ijrs=0101$ where there is no solution to the
Associativity Equation, we have $\mu=0$ and this is subsumed by the
conditions for the general case.
\end{proof}

In practice, when applying \tref{t:charnonass} the value of $\eta(-1)$
will be determined by the value of $q\bmod 4$ and a number of the
conditions will be trivially satisfied. Also, it is legitimate to
replace any occurrence of ``square'' in \tref{t:charnonass} by
``nonzero square''. This is because the arguments of $\eta$ in
\eref{35} are known to be nonzero, by \lref{31}.

It follows from \lref{l:basicquad}(iv) that $Q_{a,b}$ is maximally
nonassociative if and only if $Q_{b,a}$ is maximally nonassociative.
This is reflected in condition (10) of \tref{t:charnonass}.  It is
also easy to see that the opposite quasigroup for a maximally
nonassociative quasigroup is itself a maximally nonassociative
quasigroup (see also \lref{l:opp}). It follows then
from \lref{l:basicquad} parts (iv) and (v) that $Q_{a,b}$ is maximally
nonassociative if and only if $Q_{1-a,1-b}$ is maximally
nonassociative. Hence substituting $a\mapsto 1-a$ and $b\mapsto 1-b$
into the conditions of
\tref{t:charnonass} should produce an equivalent set of conditions.
This can be verified with some effort.

Similarly, it can be checked that making the substitution $b=a$ into
the conditions (2)--(9) of \tref{t:charnonass} yields eight conditions
which between them exclude all the possibilities for the triple
$(\eta(-1),\eta(a),\eta(a-1))$.  In this way, we demonstrate that
$Q_{a,a}$ is never maximally nonassociative.  This was already
remarked when we defined $\Sigma$ at the start of this section, but it
does provide another consistency check for \tref{t:charnonass}.

\section{Weil bound applications}\label{s:Weil}

Throughout this section, $\F$ will be a finite field of odd order $q$.
Our main goal is to use \cref{cy:Weil} to show that the conditions
of \tref{t:charnonass} are satisfied by some pair $(a,b)$, provided
$q$ is large enough (although for technical reasons, we will exclude
fields of certain small characteristics). We treat the
$q\equiv1\bmod4$ and $q\equiv3\bmod4$ cases separately. For both cases
we find it useful to assume a particular (but different) relationship
between $a$ and $b$.  Each case will begin with some preliminary
lemmas that establish conditions under which that relationship creates
the desired outcome.  We begin with the case $q\equiv1\bmod4$.

\begin{lemm}\label{l:1mod4cond}
Suppose that $q\equiv 1 \bmod 4$.  Let $a\in\F$ be such that
$a^3-a^2+2a-1$ is square, while $a$, $a-1$, $a^2+a-1$, and $a^2-3a+1$
are nonsquares.  Then $Q_{a,1-a}$ is maximally nonassociative.
\end{lemm}

\begin{proof} 
Let $b=1-a$ and note that $\eta(-1)=0$ since $q\equiv1\bmod4$.  Also
$\eta(ab)=\eta((1-a)(1-b))=0$.
Note that $\eta(a)=\eta(1-a)=1$ implies that $a\ne\{0,1\}$. It
follows that $ab$ and $(a-1)(b-1)$ are nonzero squares.

By assumption, $a,b,a-1,b-1$, $\nu=a^2-2a+b=a^2-3a+1=b^2-a$ and
$\mu=b^2-2b+a=a^2+a-1=a^2-b$, are nonsquares while $-1$ and
$\sigma=a^2b-a^2-ab+b=-a^3+a^2-2a+1$ are squares.  It follows that 
$a^2\ne b$, $b^2\ne a$ and that all of the
conditions of \tref{t:charnonass} are met.
\end{proof}

\begin{lemm}\label{l:aresqfree1mod4}
The list of polynomials
$x$, $x-1$, $x^2+x-1$, $x^2-3x+1$ and $x^3-x^2+2x-1$ 
avoids squares over any field $\F$ with $\chr(\F)\notin\{2,5\}$.
\end{lemm}

\begin{proof}
It is trivial to check that none of the nonlinear polynomials can
share a root with either of the linear polynomials. So we can ignore
the linear polynomials when it comes to looking for a subset of the
polynomials that might multiply to give a perfect square. That leaves
only one polynomial of odd degree, which can therefore also be ruled out.
It is routine to check that the two quadratics between them
have four distinct roots in $\overline\F$ provided $\chr(\F)\notin\{2,5\}$. 
\end{proof}

Armed with these preliminary lemmas we can now show the existence of
maximally nonassociative quasigroups in large fields of order
$q\equiv1\bmod 4$ (except for fields of characteristic $5$).

\begin{thm}\label{t:1mod4exist}
Let $\F$ be a field of prime power order $q\equiv 1\bmod 4$ such that
$\chr(\F)\ne 5$. If $q\ne17$, then there exists $a\in \F$ such
that $Q_{a,1-a}$ is a maximally nonassociative quasigroup.
\end{thm}

\begin{proof} 
It is enough to find $a\in \F$ that fulfils the conditions
of \lref{l:1mod4cond}. By \lref{l:aresqfree1mod4} the number of such
elements can be estimated by \tref{t:Weil}.  In terms of
\cref{cy:Weil} we have $k=5$ and $d_1=d_2=1$, $d_3=d_4=2$ and $d_5=3$,
so that $D=9$.  Hence it suffices if $q/32>(\sqrt
q+1)9/2-\sqrt{q}(31/32)$, which is true if $q\ge13056$. For all
smaller $q$ we use direct computation (the results of which
are given at \cite{WWWW}).  For prime powers
$q\equiv1\bmod4$ in the range $9\le q<13056$ that are not powers of 5
there is $a\in\F_q$ satisfying \lref{l:1mod4cond} except for
$q\in\{17,37,49\}$.
For $q=37$, $Q_{18,20}$ is maximally nonassociative
and for $q=49$, $Q_{3t,1-3t}$ is maximally nonassociative over
$\Z_7[t]/(t^2+t+3)$ (note that the
conditions in \lref{l:1mod4cond} are sufficient but not necessary).
\end{proof}

In $\F_{17}$ there are no maximally nonassociative quasigroups of the
form $Q_{a,1-a}$; however $Q_{4,8}$ is maximally nonassociative.
Although characteristic 5 fields are excluded from \tref{t:1mod4exist},
there are quadratic orthomorphisms of $\F_{25}$ that produce 
maximally nonassociative quasigroups. For example, we can take 
$Q_{t,t^3}$ in $\Z_5[t]/(t^2+t+2)$. Note also that
\cite{dl} gives a construction for order $p^2$ for all odd primes $p$.

It is worth noting that the quasigroup $Q_{a,1-a}$ involved
in \tref{t:1mod4exist} is isomorphic to its opposite quasigroup, by
\lref{l:basicquad}. This makes \tref{t:charnonass}
easier to satisfy since many of the conditions coincide.  The same
approach does not work for $q\equiv3\bmod4$ since in that case
$Q_{a,1-a}$ is actually equal to its opposite by \lref{l:basicquad},
which is the same as saying that $Q_{a,1-a}$ is commutative.  It is
not possible for a commutative quasigroup of order $n>1$ to be
maximally nonassociative since $x*(y*x)=(y*x)*x=(x*y)*x$ for all $x,y$
in a commutative quasigroup, which means that there will always be at
least $n^2$ associative triples. Hence we need a different approach for
$q\equiv3\bmod4$.

\begin{lemm}\label{l:3mod4cond}
Suppose that $\F$ is a field of order $q\equiv 3 \bmod 4$ and
$\chr(F)>19$.  Let $a\in\F_q$ be such that $a$, $a-1$, $a+2$, $4a-1$
and $16a-7$ are squares, while $a-4$, $4a-3$, $4a+3$ and $16a-1$ are
nonsquares.  
Then $Q_{a,4a}$ is maximally nonassociative.
\end{lemm}

\begin{proof} 
Let $b=4a$ and note that $\eta(-1)=1$ since $q\equiv3\bmod4$.  Also
$\eta(ab)=\eta(4a^2)=0$ and $\eta((a-1)(b-1))=\eta(a-1)+\eta(4a-1)=0$.
Since we are insisting that $\eta(4a-1)=0$ and
$\eta(4a-3)=\eta(4a+3)=1$, we know that $a\ne\{0,1,1/4\}$. It
follows that $ab$ and $(a-1)(b-1)$ are nonzero squares.

Next we consider condition (1) in \tref{t:charnonass}.
If $4a=b=a^2$ then $a=4$ (since $a\ne0$). If in addition $0=2b-b^2-a=-228$
then $\F$ must have characteristic at most 19, which we are assuming is not
the case. Similarly, we cannot have $4a=b=2a-a^2$ unless $a=-2$ which together
with $a=b^2$ would force $0=66$. Hence the restriction $\chr(\F)>19$ ensures
that we can ignore the condition (1) in \tref{t:charnonass} (and its image
under interchange of $a$ and $b$).

By assumption, $a,b,a-1,b-1$,
$\mu=b^2-2b+a=16a^2-7a=a(16a-7)$, $b-a^2=a(4-a)$,
$\nu=a^2-2a+b=a(a+2)$ and
$a-b^2=a(1-16a)$ are all squares, while $-1$,
$ab-a+b=a(4a+3)$ and $ab-b+a=a(4a-3)$ are nonsquares.
It follows that all of the conditions of \tref{t:charnonass} are met.
\end{proof}

\begin{lemm}\label{l:aresqfree3mod4}
If $\chr(\F)>19$, then the list of polynomials
$x$, $x-1$, $x+2$, $4x-1$, $16x-7$, $x-4$, $4x-3$, $4x+3$ and $16x-1$
avoids squares.
\end{lemm}

\begin{proof} 
The roots $0,1,-2,1/4,7/16,4,3/4,-3/4$ and $1/16$ of these linear polynomials
are all distinct when $\chr(\F)>19$. It follows that the list of polynomials
avoids squares.
\end{proof}

Again these preliminary lemmas will now allow us to show the existence of
maximally nonassociative quasigroups in fields of large order
$q\equiv3\bmod 4$ and large characteristic.

\begin{thm}\label{t:3mod4exist}
Let $\F$ be a field of prime power order $q\equiv 3\bmod 4$ such that
$\chr(\F)>19$. If $q\ne79$, then there exists $a\in \F$ such
that $Q_{a,4a}$ is a maximally nonassociative quasigroup.
\end{thm}

\begin{proof} 
It is enough to find $a\in \F$ that fulfils the conditions
of \lref{l:3mod4cond}. By \lref{l:aresqfree3mod4} the number of such
elements can be estimated by \tref{t:Weil}.  In terms of
\cref{cy:Weil} we have $k=9$ and $d_1=\dots=d_9=1$.  Hence it suffices if
$q/512>(\sqrt q+1)9/2-\sqrt{q}(511/512)$, which is 
true if $q\ge3219456$. For smaller orders, again we do a direct computation
\cite{WWWW}.
For $1663<q<3219456$ we found $a\in\F$ satisfying \lref{l:3mod4cond}.
For $19<q\le 1663$ we found $a\in\F$ for which $Q_{a,4a}$
is a maximally nonassociative quasigroup, unless $q=79$.
\end{proof}

Although $\F_{79}$ allows no maximally nonassociative quasigroup of
the form $Q_{a,4a}$, it does allow the maximally nonassociative
quasigroup $Q_{10,26}$. Also $Q_{5,6}$ is a maximally nonassociative
quasigroup when $q={19}$.

\section{Additional orthomorphisms}\label{s:specorth}

In this final section we will wrap up the proof of \tref{t:main}.  To
do that we will need to construct maximally nonassociative quasigroups
of certain small orders.  First we give an orthomorphism $\psi$ of the
cyclic group $\Z_n$ which produces a maximally nonassociative
quasigroup, via \eref{:quasiorth}, for orders
$n\in\{21,33,35,55\}$. We present each orthomorphism as a permutation
in cycle notation. In all but the first case this permutation is an
involution. There are no involutions in $\Z_{21}$ which work.
\begin{align*}
\Z_{21}:\ & (1, 2)(3, 8, 17, 13, 19, 9, 6, 4, 12, 16, 10, 20, 15, 18,
11)(5, 7, 14) \\ \Z_{33}:\ &(1, 2)(3, 5)(4, 15)(6, 23)(7, 14)(8,
29)(9, 27)(10, 19)(11, 21)(12, 32)(13, 16)\\&(17, 31)(18, 26) (20, 
24)(22, 28)(25, 30) \\ \Z_{35}:\ &(1, 2)(3, 5)(4, 8)(6, 16)(7, 19)(9,
33)(10, 15)(11, 24)(12, 30)(13, 28)(14, 23)\\&(17, 31)(18, 34) (20, 
27)(21, 29)(22, 25)(26, 32) \\ \Z_{55}:\ & (1, 9)(2, 44)(3, 35)(4,
15)(5, 25)(6, 33)(7, 41)(8, 22)(10, 20)(11, 37) (12, 27)\\ &(13, 32)(14,
31)(16, 38) 
(17, 18)(19, 28)(21, 23)(24, 36)(26, 50)(29, 54)(30,
46) \\ &
(34, 52)(39, 43)(40, 45)(42, 49)(47, 53)(48, 51)
\end{align*}
For even orders, there are no orthomorphisms of the cyclic group, so
we need to use noncyclic groups. In the following permutations we omit
commas within cycles and also adopt a shorthand notation for group
elements. We write $(a,b)$ as $a_b$ and $(a,b,c,d)$ as $a_{bcd}$.  In
this way we present orthomorphisms which produce a maximally
nonassociative quasigroup, via \eref{:quasiorth}, for orders $n\in\{16,20,24,28,32\}$:
\begin{align*}
\Z_{8}\times\Z_2:\ & (0_11_04_17_15_02_1)(1_16_15_17_03_13_04_06_02_0)
\\ \Z_{10}\times\Z_2:\ &
(0_14_0)(1_01_16_1)(2_04_12_19_16_05_13_05_08_19_0)(3_17_1)(7_08_0)
\\ \Z_{12}\times\Z_2:\ &
(0_13_17_07_16_12_03_01_16_08_05_09_010_11_0)(2_14_010_05_111_08_14_111_19_1)
\\ \Z_{14}\times\Z_2:\ &
(0_112_12_03_01_19_08_011_17_14_15_02_111_013_16_113_05_19_13_18_14_06_012_0
 \\ & \ 7_010_010_11_0)
\\ \Z_{4}\times\Z_2\times\Z_2\times\Z_2:\ &
(0_{001}3_{110}0_{010}0_{111}2_{000})
(0_{011}2_{011}3_{001}2_{001}2_{010}1_{110}1_{000})
(0_{100}3_{010})\\&
(0_{101}1_{010}3_{000}2_{101}2_{111}3_{111}2_{110})
(0_{110}1_{011}1_{100}) \\ &
(1_{001}1_{101}3_{011}1_{111}2_{100}3_{101}3_{100})
\end{align*}

We are now finally in a position to complete the proof
of \tref{t:main}:  

\begin{proof}
The strategy is to find a suitable factorisation $n=f_1f_2\cdots f_m\ge9$
where $f_1\ge f_2\ge\cdots\ge f_{m}>2$ and there exists a maximally
nonassociative quasigroup of order $f_1$. We will then be able to
iteratively use \cref{cy:prod} to produce maximally nonassociative
quasigroups of order $f_1f_2$, $f_1f_2f_3,\dots,f_1f_2\cdots f_m$.

Start with a factorisation of $n$ into primes. We then modify the 
factorisation by taking the following steps in the given order:
\begin{itemize}
\item[(1)] Repeatedly replace pairs of factors that both equal $2$ 
by a single factor equal
to $4$, until there is at most one factor that equals $2$.
\item[(2)] If the largest factor is currently at most 11 then
look for two or three factors whose product is in
$\{9,16,20,21,24,25,28,32,33,35,49,55,121\}$ and replace
those factors by their product. If there are several options then
choose one with the largest product.
\item[(3)] If there is a factor that equals 2, combine it with the next smallest
factor.
\item[(4)] Sort the factors into weakly decreasing order.
\end{itemize}

It is possible that step (2) may fail; however that only happens if
$$n\in\{10,11,12,14,15,22,30,44,77,88,154\}$$ and these cases are excluded from
\tref{t:main}.  Assume that step (2) works and that $n=f_1f_2\cdots
f_m$ is the factorisation that we arrive at after step (4).  By
design, $f_1\ge f_2\ge\cdots\ge f_{m}>2$.  Thus it suffices to find a
maximally nonassociative quasigroup of order $f_1$.

If $f_1$ was created by step (4) then $n\in\{40,42,56,66,90,110\}$
or $n=2p_1$ or $n=2p_1p_2$ for odd primes $p_1,p_2$ with $p_1\le p_2<2p_1$.
These cases are all excluded from \tref{t:main}. So we may assume
that $f_1$ was not created by step (4). That means that $f_1$ is a
prime larger than 11 or was created in step (2).

In \tref{t:1mod4exist}, \tref{t:3mod4exist} and the surrounding
comments we showed that there exists a maximally nonassociative
quasigroup of order $p$ for any prime $p>11$ as well as of orders
$3^2,5^2,7^2$ and $11^2$.  We have also given explicit examples of
order 16, 20, 21, 24, 28, 32, 33, 35 and 55 earlier in this section.
Thus in all cases there is a maximally nonassociative quasigroup of
order $f_1$, which completes the proof.
\end{proof}

We finish by describing some questions raised by our work. The
most obvious is to resolve the possible exceptions in \tref{t:main}.
Another is to estimate the asymptotic proportion of quadratic
orthomorphisms that satisfy the conditions
of \tref{t:charnonass}. Numerical experiments suggest that roughly 1/8
of quadratic orthomorphisms work when $q\equiv1\bmod4$, whereas the
proportion for $q\equiv3\bmod4$ is closer to 1/20. The true asymptotic
proportions have been established in a follow-up paper \cite{DW2}.

Another direction for research is to consider how few associative triples
a quasigroup can achieve when it is not idempotent. Both \cite{gh} and 
\cite{dv1} give lower bounds for the number of associative triples in this case.
It remains to be determined whether these bounds are achieved and for
what orders.

\def\sym{\mathcal{S}}

Our final research direction concerns the symmetry groups of maximally
nonassociative quasigroups. The \emph{automorphism group} of a
quasigroup of order $n$ is its stabiliser under the natural action of
the symmetric group $\sym_n$. Its \emph{autoparatopism group} is its
stabiliser under the action of $\sym_n\wr\sym_3$. Orbits under these
actions are called \emph{isomorphism classes} and \emph{species}
respectively.  Examples built directly from quadratic orthomorphisms
have a very large automorphism group, by \lref{l:basicquad}. For
example, there are 12 quadratic orthomorphisms of $\F_{27}$ that
produce maximally nonassociative quasigroups. These form four
isomorphism classes, which come from only two different species.
Representatives of each species have automorphism group of order $351$
(the minimum order possible, by \lref{l:basicquad}) and autoparatopism
groups of order 702 and 1053, respectively. In contrast, employing the
product construction in \tref{22} can destroy all symmetry. For
example, suppose that we take $Q$ to be the unique maximally
nonassociative quasigroup of order 9 (which has automorphism group of
order 72 and autoparatopism group of order 432, see \cite{dv2}) and
take $U$ to be the unique idempotent quasigroup of order $3$. There
are $9\times8\times7=504$ choices for the injection $j$, and these
produce 17 isomorphism classes of maximally nonassociative quasigroups
of order 27.  Examining representatives of the 17 classes, we find one
with an automorphism group of order 6, another with an automorphism
group of order 3, three with an automorphism group of order 2 and
twelve with trivial automorphism group. All 17 representatives come
from different species and have autoparatopism group equal to their
automorphism group. In light of these observations, we ask what is the
smallest order of a maximally nonassociative quasigroup with trivial
automorphism group? Also it would be interesting to understand what
automorphisms/autoparatopisms a maximally nonassociative quasigroup
can have.

\subsection*{Supplementary data}

Data that was generated in the computational proofs of \lref{l:tech},
\tref{t:1mod4exist} and \tref{t:3mod4exist} for this article is
available on the second author's homepage \cite{WWWW}.  In addition, a
number of examples of maximally nonassociate quasigroups are given
there.

\subsection*{Acknowledgement}
This work was supported in part by Australian Research Council grant
DP150100506.  We thank Petr Lison\v ek for keeping us informed
regarding his paper~\cite{Lis20} and Gabriel Verret for posing the
question to us as to whether a maximally nonassociative quasigroup can
have trivial automorphism group.


\begin{thebibliography}{99}

\bibitem{dl} A.~Dr\'apal and P.~Lison\v ek, 
Maximal nonassociativity via nearfields, 
\emph{Finite Fields Appl.} \textbf{62} (2020) 101610.

\bibitem{dv0} A.~Dr\'apal and V.~Valent, 
Few associative triples, isotopisms and groups, 
\emph{Des. Codes Cryptogr.} \textbf{86} (2018), 555--568.

\bibitem{dv1} A.~Dr\'apal and V.~Valent, 
High nonassociativity in order 8 and an associative index estimate, 
\emph{J.~Combin.~Des.} \textbf{27} (2019), 205--228.

\bibitem{dv2} A.~Dr\'apal and V.~Valent,
Extreme nonassociativity in order nine and beyond, 
\emph{J.~Combin.~Des.} \textbf{28} (2020), 33--48.

\bibitem{DW2} A.~Dr\'apal and I.\,M.~Wanless,
On the number of quadratic orthomorphisms that produce maximally
nonassociative quasigroups, arXiv:2005.11674.


\bibitem{Eva18}
A.\,B.~Evans,
\emph{Orthogonal Latin squares based on groups}, 
Develop. Math. {\bf57}, Springer, Cham, 2018.

\bibitem{gh} O.~Gro\v sek and P.~Hor\' ak, 
\emph{On quasigroups with few associative triples}, 
Des. Codes Cryptogr. \textbf{64} (2012), 221--227.

\bibitem{Kep80} T.~Kepka, 
A note on associative triples of elements in cancellation groupoids, 
\emph{Comment. Math. Univ.~Carolin.} \textbf{21} (1980), 479--487.
\url{https://dml.cz/handle/10338.dmlcz/106014}

\bibitem{kr} A.~Kotzig and C.~Reischer, 
Associativity index of finite quasigroups, 
\emph{Glas. Mat. Ser. III} \textbf{18} (1983), 243--253.

\bibitem{Lis20}
P.~Lison\v ek, Maximal nonassociativity via fields,
\emph{Des. Codes Cryptogr.} \textbf{88} (2020), 2521--2530.

\bibitem{ARL}
A.~Rojas-Le\'on,
More general exponential and character sums,
In: \emph{Handbook of Finite Fields}, G.L.~Mullen and D.~Panario (Eds.), 
CRC Press, 2013, pp161--169.

\bibitem{diagcyc}
I.\,M.~Wanless, Diagonally cyclic Latin squares,
{\it European J.\ Combin.\/} {\bf25} (2004), 393--413.

\bibitem{cyclatom} I.\,M.~Wanless,
Atomic Latin squares based on cyclotomic orthomorphisms,
{\it Electron.\ J.\ Combin.\/}, {\bf12} (2005), R22.

\bibitem{WWWW}
I.\,M.~Wanless, Author homepage,
\url{http://users.monash.edu.au/~iwanless/data}.

\end{thebibliography}
\end{document}